\documentclass[11pt,reqno]{amsart}
\usepackage{amsmath,amssymb,mathrsfs}
\usepackage{graphicx,subfigure,cite,times}
\usepackage[
colorlinks,
citecolor=blue,
linkcolor=blue,
backref=page
allcolors=black
]{hyperref}
\usepackage{color}
\usepackage[all]{xy}
\usepackage{xcolor}
\newtheorem{theorem}{Theorem}[section]

\newtheorem{lemma}[theorem]{Lemma}

\newtheorem{remark}[theorem]{Remark}

\newtheorem{definition}[theorem]{Definition}
\newtheorem{example}[theorem]{Example}
\numberwithin{equation}{section}

\allowdisplaybreaks

\begin{document}
\large
\title[KAM theorem on modulus of continuity about parameter]{KAM theorem on modulus of continuity about parameter}
\author{Zhicheng Tong}
\address{\scriptsize (Z. C. Tong)~College of Mathematics, Jilin University, Changchun 130012, P. R. China.}
\email{tongzc20@mails.jlu.edu.cn}

\author{Jiayin Du}
\address{\scriptsize (J. Y. Du)~College of Mathematics, Jilin University, Changchun 130012, P. R. China.}
\email{dujy20@mails.jlu.edu.cn}

\author{Yong Li}
\address{\scriptsize  (Y. Li)~Institute of Mathematics, Jilin University, Changchun 130012, P.R. China. School of Mathematics and Statistics, Center for Mathematics and Interdisciplinary Sciences, Northeast Normal University, Changchun, Jilin 130024, P. R. China.}
\email{liyong@jlu.edu.cn}
\thanks{The corresponding author (Y. Li) was supported in part by National Basic Research Program of China Grant (2013CB834100) and NSFC Grant (12071175, 11171132,  11571065), Project of Science and Technology Development of Jilin Province (2017C028-1, 20190201302JC), and Natural Science Foundation of Jilin Province (20200201253JC)}



\begin{abstract}
	In this paper, we study the Hamiltonian systems $ H\left( {y,x,\xi ,\varepsilon } \right) = \left\langle {\omega \left( \xi  \right),y} \right\rangle  + \varepsilon P\left( {y,x,\xi ,\varepsilon } \right) $, where $ \omega $ and $ P $ are continuous about $ \xi $. We prove that persistent invariant tori possess the same frequency as the unperturbed tori, under certain transversality condition and weak convexity condition for the frequency mapping $ \omega $. As a direct application, we prove a KAM theorem when the perturbation $P$ holds arbitrary H\"{o}lder continuity with respect to parameter $ \xi $. The infinite dimensional case is also considered. To our knowledge, this is the first approach to the systems with the only continuity in parameter beyond H\"older's type.
\end{abstract}

\maketitle
\tableofcontents

\maketitle
\tableofcontents
\section{Introduction}\label{sec-1}
The KAM method has been well studied since \cite{MR0140699,MR0163025,MR0170705,MR0199523,MR0206461,MR0147741}. Further efforts have been made in parameterized KAM \cite{R-18,MR1001032,MR0668410,MR1022821}, finitely differentiable KAM \cite{MR4104457,MR2111297}, Gevrey smooth KAM \cite{MR2104602}, general resonance KAM \cite{MR2003447}, generalized Hamiltonian KAM \cite{MR1926285}, iso-manifold KAM \cite{MR4355926}, R\"ussmann non-degenerate KAM \cite{MR1938331,MR1302146,MR1483538,MR1843664,MR1354540}, multiscale KAM \cite{MR3870557,MR4066039} and so on. However, no one knows how many dynamics of Hamiltonian systems starting from a weak smooth manifold can be maintained, and we will touch this question.

In this paper, we are able to prove that the perturbed invariant tori have the same Diophantine frequency as unperturbed quasi-periodic tori for a family of Hamiltonian systems with continuous parameters. More precisely, we consider the Hamiltonian systems under small perturbations:
\begin{equation}\label{H}
	H\left( {y,x,\xi ,\varepsilon } \right) = \left\langle {\omega \left( \xi  \right),y} \right\rangle  + \varepsilon P\left( {y,x,\xi ,\varepsilon } \right),
\end{equation}
where $ x $ is the angle variable in the standard torus $ \mathbb{T}^n $, $ n $ refers to the dimension; $ y $ is the action variable in $ G\subset {\mathbb{R}^n} $; and $ \xi $ is a parameter in $ \mathcal{O} \subset {\mathbb{R}^n} $, where $ G,\mathcal{O} $ are bounded connected closed regions with interior points. Moreover, $ \omega \left(  \cdot  \right) $ is continuous about $ \xi $ on $ \mathcal{O} $;  $ P\left( { \cdot , \cdot ,\xi ,\varepsilon } \right) $ is real analytic about $ y $ and $ x $ on $ G \times {\mathbb{T}^n} $, and $ P\left( {y,x, \cdot ,\varepsilon } \right) $ is only continuous about parameter $ \xi $; $ \varepsilon>0 $ is a sufficiently small scalar.   Such Hamiltonian systems are called usually parameterized ones. As is well known, at least Lipschitz continuity about parameter is needed in all KAM type results. See, for instance, \cite{MR1022821,MR1401420}. Hence one naturally asks the following basic problems:

\begin{itemize}
\item[\textbf{(Q1)}] \textbf{Can the Lipschitz type condition on parameter $\xi$ be weakened to the H\"{o}lder, further continuous one?}

\item[\textbf{(Q2)}] \textbf{How much persistence  
    of invariant tori is there for the action variables or parameters on a general manifold (might be zero Lebesgue measure), further frequency-preserving?}

\item[\textbf{(Q3)}] \textbf{What does the persistence in the infinite dimensional cases?}
\end{itemize}


These are certainly nontrivial. From the usual KAM iteration process, to overcome small divisor, one has to dig out some parameter domain at each KAM step,  while the Lipschitz continuity exactly ensures that the rest parameter set remains a positive measure. In the present paper, we will try to answer the above-mentioned problems. Concretely, we construct some sufficient conditions in terms of the transversality condition (\textbf{A1}), as well as the weak convexity condition (\textbf{A2}) for the frequency mapping, see Section \ref{sec-3} for details. It is worth mentioning that (\textbf{A2}) is indispensable, otherwise there will be counter examples. 
Recently, the restriction has been weakened to the H\"{o}lder continuous case, see \cite{R-18}. Fortunately, here we can further extend it to the case of continuous dependence with respect to parameter. Due to weak regularity, we cannot employ traditional way of digging out domain, which leads to another more direct manner to find out appropriate parameters. Particularly, we also develop such approach to infinite dimensional Hamiltonian systems.

This paper is organized as follows: in Section \ref{sec-2}, we introduce some basic notions; our main KAM Theorem \ref{theorem1} with frequency-preserving under weak regularity is stated in Section 3, and the proof is shown in Sections \ref{sec-4} and \ref{sec-5};  in Section \ref{sec-7}, we first make some comments, and then similarly give an infinite dimensional version about full dimensional tori; finally in Section \ref{sec-6}, we prove an arbitrary H\"older continuous KAM by applying  Theorem \ref{theorem1}, and we also give an explicit example via weak regularity (both the frequency mapping and the perturbation are nowhere H\"older continuous on some region), which cannot be studied by any KAM theorem known.

\section{Preliminaries}\label{sec-2}

We first give some notions, including modulus of continuity as well as the norm based on it.

\begin{definition}\label{definition1}
	Denote by $ \varpi (x) $ a modulus of continuity, which is a strictly monotonic increasing continuous function on $ \mathbb{R}^+ $, such that $ \mathop {\lim }\limits_{x \to {0^ + }} \varpi \left( x \right) = 0 $ and $ \mathop {\overline {\lim } }\limits_{x \to {0^ + }} x/{\varpi }\left( x \right) <  + \infty  $.
\end{definition}
\begin{remark}\label{logarithmic Lipschitz}
	A modulus of continuity $ \varpi $ is called  Logarithmic Lipschitz if $ {\varpi  }  \sim  -1/\ln x $ with respect to $ x \to 0^+ $. This will be used in establishing KAM theorem of arbitrary H\"{o}lder continuous type.
\end{remark}
\begin{definition}
	A function $ f(x) $ is called $ \varpi  $ continuous about $ x $, if for any fixed $ x $, there exists a modulus of continuity $ \varpi $  such that the following holds
	\[\left| {f\left( x \right) - f\left( y \right)} \right| \leqslant  {\varpi}\left( {\left| {x - y} \right|} \right)\;\; \forall 0<|x-y| \leqslant 1.\]
\end{definition}

\begin{definition}\label{qr}
	Let $ {\varpi _1} $ and $ {\varpi _2} $ be modulus of continuity. We say $ {\varpi _1} $ is not weaker than $ {\varpi _2} $ if $ \mathop {\overline {\lim } }\limits_{x \to {0^ + }} {\varpi _1}\left( x \right)/{\varpi _2}\left( x \right) <  + \infty  $, and denote $ {\varpi _1} \lesssim {\varpi _2} $ (or $ {\varpi _2} \gtrsim  {\varpi _1}$).
\end{definition}
\begin{remark}\label{666}
	Let $ \bar \Omega  \subset {\mathbb{R}^d} $ be a bounded connected closed region with $ d \in {\mathbb{N}^ + } $. If $ f\left( x \right) \in C\left( {\bar \Omega } \right) $, then $ f $ has a modulus of continuity $ \varpi $.
\end{remark}

\begin{definition}
	For the perturbation function $ P\left( {y,x,\xi,\varepsilon } \right) $, which is analytic about $ y $ and $ x $ on a closed connected region $ D \subset {\mathbb{R}^n} \times {\mathbb{R}^n} $ and $ \varpi _1 $ continuous about $ \xi $  on $ \mathcal{O} $, we define its norm as follows
	\[{\left\| P \right\|_D}: = {\left| P \right|_D} + {\left\| P \right\|_{{\varpi _1}}},\]
	where
	\begin{align*}
		{\left| P \right|_D} &:= \mathop {\sup }\limits_{\xi  \in \mathcal{O}} \mathop {\sup }\limits_{\left( {y,x} \right) \in D} \left| P \right|,\\
		{\left\| P \right\|_{{\varpi _1}}} &:= \mathop {\sup }\limits_{\left( {y,x} \right) \in D} \mathop {\sup }\limits_{\xi ,\zeta  \in \mathcal{O},0 < \left| {\xi  - \zeta } \right| \leqslant {1 }} \frac{{\left| {P\left( {y,x,\xi ,\varepsilon } \right) - P\left( {y,x,\zeta ,\varepsilon } \right)} \right|}}{{{\varpi _1}\left( {\left| {\xi  - \zeta } \right|} \right)}}.
	\end{align*}
\end{definition}

\section{Statement of results}\label{sec-3}
We are now ready to state our assumptions. Mainly, for any $ \varepsilon  > 0 $ small enough, we consider the parameterized family of the perturbed Hamiltonian systems \eqref{H}, i.e.,
\begin{equation*}
	\left\{ \begin{gathered}
		H:G \times {\mathbb{T}^n} \times \mathcal{O} \to {\mathbb{R}^1}, \hfill \\
		H\left( {y,x,\xi,\varepsilon } \right) = \left\langle {\omega \left( \xi  \right),y} \right\rangle  + \varepsilon P\left( {y,x,\xi ,\varepsilon } \right). \hfill \\
	\end{gathered}  \right.
\end{equation*}
According to Remark \ref{666}, there is a modulus of continuity $ \varpi_1 $ such that $ P $ is $ \varpi_1 $ continuous about $ \xi $. 

First, we make the following assumptions:


(\textbf{A1})\ Given $ \tt{p} \in \mathbb{R}^n $ and denote by $ \mathcal{O}^o $ the interior of $ \mathcal{O} $. Assume there exists a $ {\xi _0} \in {\mathcal{O}^o} $ such that $ \omega(\xi_0)=\tt{p} $,
and for given $ \tau  > 0$, $ \omega \left( {{\xi _0}} \right) $ satisfies the Diophantine condition
\begin{equation}\label{dioph}
	\left| {\left\langle {k,\omega \left( {{\xi _0}} \right)} \right\rangle } \right| \geqslant \gamma {\left| k \right|^{ - \tau }},\;\;0 \ne k \in {\mathbb{Z}^n},\;\;\gamma  > 0.
\end{equation}

(\textbf{A2})\ Assume there exist a neighborhood $ B\left( {{\xi _0},\delta } \right) \subset \mathcal{O}^o $ of $ \xi_0 $ with some $ \delta>0 $ and a modulus of continuity $ \varpi_2 $ satisfying $ \varpi_1 \lesssim \varpi_2 $, such that
\begin{equation}\notag
	\left| {\omega \left( \xi  \right) - \omega \left( \zeta  \right)} \right| \geqslant {\varpi _2}\left( {\left| {\xi  - \zeta } \right|} \right),\;\;\forall 0 < \left| {\xi  - \zeta } \right| \leqslant {1 },\;\;\xi ,\zeta  \in B\left( {{\xi _0},\delta } \right).
\end{equation}


Then we have the following main result:

\begin{theorem}\label{theorem1}
	Consider Hamiltonian systems \eqref{H}. Assume, besides the continuity in $\xi$ and the analyticity in $(y,x)$ mentioned in Section \ref{sec-1}, that (\textbf{A1}), (\textbf{A2}) 
	hold. Then there exists a sufficiently small $ {\varepsilon _0} > 0 $, for any $ 0 < \varepsilon  < {\varepsilon _0} $,  there exists $ {\xi^* } \in B\left( {{\xi _0},\delta } \right) $, such that the	perturbed Hamiltonian system $ H\left( {y,x,{\xi^* },\varepsilon } \right) $ admits an invariant torus with frequency $ \tt{p}=\omega \left( {{\xi _0}} \right) $.
\end{theorem}
\begin{remark}\label{Remarkk3.2}
	Assumption (\textbf{A2}) is extremely important and indispensable, 
	otherwise there will be counter examples, see \cite{R-18}. In fact, it could be weakened, such as the modulus of continuity depends on the position of the point, and we do not pursue this point. Note that (\textbf{A2}) is assumed to be a local property and not necessarily global, because the solvability provided by (\textbf{A1}) can actually be around $ \xi_0 $, thanks to the uniform smallness of the KAM perturbation. This implies that the continuity of the frequency mapping $ \omega $ could be arbitrarily weak beyond $ B(\xi_0, \delta) $, such as nowhere differentiable or even worse. 
This is a main point of our KAM theorem.
\end{remark}

In order to prove Theorem \ref{theorem1}, we need to give another KAM iterative scheme which uses the technique of parameter translation rather than the traditional way of digging out domain.

\section{KAM step}\label{sec-4}
In this section, we will give some lemmas of the standard KAM theorem, and the proof will be omitted here. For more details, see \cite{R-18}. Without losing generality, let $ diam\mathcal{O} \leqslant 1 $.
\subsection{Description of the $ 0 $-th KAM step}

Denote $\rho  = 1/10$, and let $ \eta  > 0 $ be an integer such that $ {\left( {1 + \rho } \right)^\eta } > 2 $. We define $ \gamma  = {\varepsilon ^{1/20}} $. We first define the following $ 0 $-th KAM step parameters:
\begin{align}
	&{r_0} = r,\;{\gamma _0} = \gamma ,\;{e_0} = 0,\;{\bar h _0} = 0,\;{\mu _0} = {\varepsilon ^{\frac{1}{{40\eta \left( {\tau  + 1} \right)}}}},\;{s_0} = \frac{{s{\gamma _0}}}{{16\left( {{M^ * } + 2} \right)K_1^{\tau  + 1}}},\notag \\
	&{\mathcal{O}_0} = \left\{ {\xi  \in \mathcal{O}|\left| {\xi  - {\xi _0}} \right| < dist\left( {{\xi _0},\partial \mathcal{O}} \right)} \right\},\notag \\
	&D \left( {{s_0},{r_0}} \right) = \left\{ {\left( {y,x} \right)|\ |y| < {s_0},\left| {\operatorname{Im} x} \right| < {r_0}} \right\},\notag
\end{align}
where $ 0 < {s_0},{\gamma _0},{\mu _0} \leqslant 1,\tau  > 0 $ and $ {M^ * } > 0 $ is a constant defined as in Lemma \ref{3.3}. Therefore, we can write
\begin{align}
	{H_0}: ={}& H\left( {y,x,{\xi _0}} \right) = {N_0} + {P_0},\notag \\
	{N_0}: ={}& {N_0}\left( {y,{\xi _0},\varepsilon } \right) = {e_0} + \left\langle {\omega \left( {{\xi _0}} \right),y} \right\rangle  + {\bar h _0},\notag \\
	{P_0}: ={}& \varepsilon P\left( {y,x,{\xi _0},\varepsilon } \right).\notag
\end{align}
According to the above parameters, we have the following estimate for $ P_0 $.

\begin{lemma}\label{3.1}
	There holds
	\begin{equation}\notag
		{\left\| {{P_0}} \right\|_{D\left( {{s_0},{r_0}} \right)}} \leqslant \gamma _0^{5}s_0^4{\mu _0}.
	\end{equation}
\end{lemma}
\begin{proof}
	See Lemma 3.1 in \cite{R-18}.
\end{proof}

\subsection{Induction from $ \nu $-th KAM step}
\subsubsection{Description of the $ \nu $-th KAM step}
We now define the $ \nu $-th KAM step parameters:
\begin{equation}\label{miuniu}
	{r_\nu } = {r_{\nu  - 1}}/2 + {r_0}/4,\;\;{s_\nu } = \mu _\nu ^{2\rho }{s_{\nu  - 1}}/8,\;\;{\mu _\nu } = {8^4}\mu _{\nu  - 1}^{1 + \rho }.
\end{equation}
Now suppose that at $ \nu $-th step, we have arrived at the following real analytic Hamiltonian:
\begin{equation}\label{1}
	\begin{aligned}
		{H_\nu } ={}& {N_\nu } + {P_\nu }, \\
		{N_\nu } ={}& {e_\nu } + \left\langle {\omega \left( {{\xi _0}} \right),y} \right\rangle  + {\bar h _\nu }\left( {y,\xi } \right),
	\end{aligned}
\end{equation}
defined on $ D\left( {{s_\nu },{r_\nu }} \right) $ and
\begin{equation}\notag
	{\left\| {{P_\nu }} \right\|_{D\left( {{s_\nu },{r_\nu }} \right)}} \leqslant \gamma _0^{5}s_\nu ^4{\mu _\nu }.
\end{equation}
The equation of motion associated to $ H_\nu $ is
\begin{equation}\label{motion}
	\left\{ \begin{gathered}
		{\dot{y}_\nu } =  - {\partial _{{x_\nu }}}{H_\nu }, \hfill \\
		{\dot{x}_\nu } = {\partial _{{y_\nu }}}{H_\nu }. \hfill \\
	\end{gathered}  \right.
\end{equation}

Except for additional instructions, we will omit the index for all quantities of the present
KAM step (at $ \nu $th-step) and use + to index all quantities (Hamiltonian, domains, normal form, perturbation, transformation, etc.) in the next KAM step (at $ \nu +1$th-step). To simplify the notations, we will not specify the dependence of $ P $, $ P_+ $ etc. All the constants $ c_1 $ - $ c_6 $ below are positive and independent of the iteration process, and we will also use $ c $ to denote any intermediate positive constant which is independent of the iteration process.

Define
\begin{align}
	{r_ + } ={}& r/2 + {r_0}/4,\notag \\
	{s_ + } ={}& \alpha s/8,\alpha  = {\mu ^{2\rho }} = {\mu ^{1/{5}}},\notag \\
	{\mu _ + } ={}& {8^4}{c_0}{\mu ^{1 + \rho }},{c_0} = 1 + \mathop {\max }\limits_{1 \leqslant i \leqslant 6} {c_i},\notag \\
	{K_ + } ={}& {\left( {\left[ { - \ln \mu } \right] + 1} \right)^{3\eta }},\notag \\
	D\left( s \right) ={}& \left\{ {y \in {\mathbb{C}^n}|\left| y \right| < s} \right\},\notag \\
	\hat D ={}& D\left( {s,{r_ + } + 7\left( {r - {r_ + }} \right)/8} \right),\notag \\
	\tilde D ={}& D\left( {s/2,{r_ + } + 6\left( {r - {r_ + }} \right)/8} \right),\notag \\
	{D_{i\alpha /8}} ={}& D\left( {i\alpha s/8,{r_ + } + \left( {i - 1} \right)\left( {r - {r_ + }} \right)/8} \right),\;1 \leqslant i \leqslant 8,\notag \\
	{D_ + } ={}& {D_{\alpha /8}} = D\left( {{s_ + },{r_ + }} \right),\notag \\
	\Gamma \left( {r - {r_ + }} \right) ={}& \sum\limits_{0 < \left| k \right| \leqslant {K_ + }} {{{\left| k \right|}^{3\tau  + 5}}{e^{ - \left| k \right|\left( {r - {r_ + }} \right)/8}}}. \notag
\end{align}
\subsubsection{Construct a symplectic transformation}
We will construct a symplectic coordinate transformation $ \Phi _ +  $:
\begin{equation}\notag
	{\Phi _ + }:\left( {{y_ + },{x_ + }} \right) \in D\left( {{s_ + },{r_ + }} \right) \to {\Phi _ + }\left( {{y_ + },{x_ + }} \right) = \left( {y,x} \right) \in D\left( {s,r} \right)
\end{equation}
such that it transforms the Hamiltonian \eqref{1} into the Hamiltonian of the next KAM cycle (at $ (\nu + 1) $-th step)
\begin{equation}\notag
	{H_ + } = H \circ {\Phi _ + } = {N_ + } + {P_ + },
\end{equation}
where $ N_+ $ and $ P_+ $ have similar properties as $ N $ and $ P $ respectively on $ D(s_+,r_+) $, and the equation of motion \eqref{motion} is changed into
\begin{equation}\notag
	\left\{ \begin{gathered}
		{y_ + } =  - {\partial _{{x_ + }}}{H_ + }, \hfill \\
		{x_ + } = {\partial _{{y_ + }}}{H_ + }. \hfill \\
	\end{gathered}  \right.
\end{equation}
\subsubsection{Truncation}
Consider the truncation of Taylor-Fourier series of $ P $
\[P = \sum\limits_{k \in {\mathbb{Z}^n},l \in \mathbb{Z}_ + ^n} {{p_{kl}}{y^l}{e^{\sqrt { - 1} \left\langle {k,x} \right\rangle }}} ,\;\;R = \sum\limits_{\left| k \right| \leqslant {K_ + },\left| l \right| \leqslant 4} {{p_{kl}}{y^l}{e^{\sqrt { - 1} \left\langle {k,x} \right\rangle }}} .\]
\begin{lemma}\label{3.2}
	Assume that
	\begin{equation}\notag
		\int_{{K_ + }}^{ + \infty } {{t^n}{e^{ - t\left( {r - {r_ + }} \right)/16}}dt}  \leqslant \mu .
	\end{equation}
	Then there is a constant $ c_1>0 $ such that
	\begin{align}
		{\left\| {P - R} \right\|_{{D_\alpha }}} \leqslant{}& {c_1}\gamma _0^{5}{s^4}{\mu ^2}, \notag \\
		{\left\| R \right\|_{{D_\alpha }}} \leqslant{}& {c_1}\gamma _0^{5}{s^4}\mu .\label{R}
	\end{align}
\end{lemma}
\begin{proof}
	See Lemma 3.2 in \cite{R-18}.
\end{proof}

\subsubsection{Homological Equation}
We shall construct a symplectic transformation as the time $ 1 $-map $ \phi _F^1 $ of the flow generated by a Hamiltonian $ F $ to eliminate all resonant terms in $ R $. Define
\begin{equation}\notag
	F = \sum\limits_{0 < \left| k \right| \leqslant {K_ + },\left| l \right| \leqslant 4} {{f_{kl}}{y^l}{e^{\sqrt { - 1} \left\langle {k,x} \right\rangle }}} ,\;\;[R] = {\left( {2\pi } \right)^{ - n}}\int_{{\mathbb{T}^n}} {R\left( {y,x} \right)dx} ,
\end{equation}
then we obtain
\[	\left\{ {N,F} \right\} + R - \left[ R \right] = 0,\]
i.e.,
\begin{equation}\label{tdfc}
	\sqrt { - 1} \left\langle {k,\omega \left( {{\xi _0}} \right) + {\partial _y}\bar h } \right\rangle {f_{kl}} = {p_{kl}},\;\;\left| l \right| \leqslant 4,\\
	\;0 < \left| k \right| \leqslant {K_ + }.
\end{equation}
\begin{lemma}\label{3.3}
	Assume that
	\begin{equation}\notag
		\mathop {\max }\limits_{\left| i \right| \leqslant 2} {\left\| {\partial _y^i\bar h  - \partial _y^i{{\bar h }_0}} \right\|_{D\left( s \right)}} \leqslant \mu _0^{1/2},\;\;s < \frac{{{\gamma _0}}}{{\left( {2\left( {{M^ * } + 2} \right)K_ + ^{\tau  + 1}} \right)}},
	\end{equation}
	where
	\[{M^ * } = \mathop {\max }\limits_{\left| i \right| \leqslant 2,y \in D\left( s \right)} \left| {\partial _y^i{{\bar h }_0}\left( {{\xi _0},y} \right)} \right|.\]
	Then the quasi-linear equations \eqref{tdfc} can be uniquely solved on $ D(s) $ to obtain a family of functions $ f_{kl} $ which are analytic in $ y $, and satisfy the following properties:
	\begin{equation}\notag
		{\left\| {\partial _y^i{f_{kl}}} \right\|_{D\left( s \right)}} \leqslant {c_2}{\left| k \right|^{\left( {\left| i \right| + 1} \right)\tau  + \left| i \right|}}\gamma _0^{4 - \left| i \right|}{s^{4 - \left| i \right|}}\mu {e^{ - \left| k \right|r}}
	\end{equation}
	for $ \left| l \right| \leqslant 4,0 < \left| k \right| \leqslant {K_ + }$ and $\left| i \right| \leqslant 4 $,	where $ c_2>0 $ is a generic constant.
\end{lemma}
\begin{proof}
	See Lemma 3.3 in \cite{R-18}.
\end{proof}

Applying the above transformation $ \phi _F^1 $ to Hamiltonian $ H $ we obtain that
\[H \circ \phi _F^1 = \left( {N + R} \right) \circ \phi _F^1: = {\bar N _ + } + {\bar P _ + },\]
where
\begin{align*}
	{}&{\bar N _ + } = N + \left[ R \right] = {e_ + } + \left\langle {\omega \left( \xi  \right),y} \right\rangle  + \Big\langle {\sum\limits_{j = 0}^\nu  {p_{01}^j\left( \xi  \right)} ,y} \Big\rangle  + {\bar h _ + }\left( {y,\xi } \right),\\
	{}&{e_ + } = e + p_{00}^\nu ,\\
	{}&{\bar h _ + }\left( {y,\xi } \right) = \bar h \left( {y,\xi } \right) + \left[ R \right] - p_{00}^\nu  - \left\langle {p_{01}^\nu \left( \xi  \right),y} \right\rangle ,\\
	{}&{\bar P _ + } = \int_0^1 {\left\{ {{R_t},F} \right\} \circ \phi _F^tdt}  + \left( {P - R} \right) \circ \phi _F^1,\\
	{}&{R_t} = \left( {1 - t} \right)\left[ R \right] + tR.
\end{align*}
\subsubsection{Translation}
In this subsection, we construct a translation so as to keep the frequency unchanged. Consider
\[\phi :\;\;x \to x,\;\;y \to y,\;\;\widetilde \xi  \to \widetilde \xi  + {\xi _ + } - \xi ,\]
where $ {\xi _ + } $ is to be determined. Let $ {\Phi _ + } = \phi _F^1 \circ \phi  $, then
\begin{align*}
	{}&H \circ {\Phi _ + } = {N_ + } + {P_ + },\\
	{}&{N_ + } = {\bar N _ + } \circ \phi  = {e_ + } + \left\langle {\omega \left( {{\xi _ + }} \right),y} \right\rangle  + \Big\langle {\sum\limits_{j = 0}^\nu  {p_{01}^j\left( {{\xi _ + }} \right)} ,y} \Big\rangle  + {\bar h _ + }\left( {y,{\xi _ + }} \right),\\
	{}&{P_ + } = {\bar P _ + } \circ \phi .
\end{align*}
\subsubsection{Frequency-preserving
}\label{kkkk}
In this subsection, we will show that the frequency can be preserved in the iteration process. Recall 
the transversality condition (\textbf{A1}) and the weak convexity condition (\textbf{A2}). Then the frequency mapping $\omega(\xi)$ is injective on ${\mathcal{O}}$, and consequently, it is surjective from ${\mathcal{O}}$ to $\omega({\mathcal{O}})$.  Therefore it is a homeomorphism and by Nagumo's theorem, we have that the Brouwer degree $\deg(\omega,{\mathcal{O}},\omega(\xi_0))=\pm 1$. This ensures that the parameter $ {{\xi _ \nu }} $ can be found in the parameter set to keep the frequency unchanged at this KAM step. The latter assures that $ \left\{ {{\xi _\nu }} \right\} $ is a Cauchy sequence. The following lemma is crucial to our arguments.
\begin{lemma}\label{crucial}
	Assume that
	\begin{equation}\notag
	{\mathop {\sup }\limits_{\xi  \in \mathcal{O}} \left| {\sum\nolimits_{j = 0}^\nu  {p_{01}^j} } \right|} < \mu _0^{1/2}.	
	\end{equation}
	There exists $ {\xi _ + } \in {B_{c\mu }}\left( \xi  \right) \subset {\mathcal{O}^o} $ such that
	\begin{equation}\label{omega}
		\omega \left( {{\xi _ + }} \right) + \sum\limits_{j = 0}^\nu  {p_{01}^j\left( {{\xi _ + }} \right)}  = \omega \left( {{\xi _0}} \right).
	\end{equation}
\end{lemma}
\begin{proof}
	The proof will be completed by an induction on $ \nu $. We start with the case $ \nu=0 $. It is obvious that $ \omega \left( {{\xi _0}} \right) = \omega \left( {{\xi _0}} \right) $. Now assume that for some $ \nu  \geqslant 1 $ we have got
	\begin{equation}\label{4.25}
		\omega \left( {{\xi _i}} \right) + \sum\limits_{j = 0}^{i - 1} {p_{01}^j\left( {{\xi _i}} \right)}  = \omega \left( {{\xi _0}} \right),\;\;{\xi _i} \in {B_{c\mu }}\left( {{\xi _{i - 1}}} \right) \subset B\left( {{\xi _0},\delta } \right) \subset \mathcal{O}^o,\;\;1 \leqslant i \leqslant \nu .
	\end{equation}
	According to the properties of topological degree, as long as $ {\mu _0} $ is sufficiently small, we have
\[\deg \Big( {\omega \left(  \cdot  \right) + \sum\limits_{j = 0}^\nu  {p_{01}^j\left(  \cdot  \right)} ,B\left( {{\xi _0},\delta } \right),\omega \left( {{\xi _0}} \right)} \Big) = \deg \left( {\omega \left(  \cdot  \right),{B\left( {{\xi _0},\delta } \right)},\omega \left( {{\xi _0}} \right)} \right) \ne 0,\]
then there exists at least a $ {\xi _ + } \in {B\left( {{\xi _0},\delta } \right)} $ such that \eqref{omega} holds.

	Note \eqref{R} in Lemma \ref{3.2} implies that
	\[	{\left\| {p_{01}^j} \right\|_{{\varpi _1}}} < c{\mu _j},\;\;0 \leqslant j \leqslant \nu, \]
	i.e.,
	\begin{equation}\label{lxm}
		\left| {p_{01}^j\left( {{\xi _ + }} \right) - p_{01}^j\left( \xi  \right)} \right| < c{\mu _j}{\varpi _1}\left( {\left| {{\xi _ + } - \xi } \right|} \right),\;\;0 \leqslant j \leqslant \nu.
	\end{equation}
	This together with (\textbf{A2}) yields that
	\begin{align}\label{4.27}
		\left| {p_{01}^\nu \left( {{\xi _ + }} \right)} \right| ={}& \left| {\omega \left( {{\xi _ + }} \right) - \omega \left( \xi  \right) + \sum\nolimits_{j = 0}^{\nu  - 1} {\left( {p_{01}^j\left( {{\xi _ + }} \right) - p_{01}^j\left( \xi  \right)} \right)} } \right|\notag \\
		\geqslant{}& \left| {\omega \left( {{\xi _ + }} \right) - \omega \left( \xi  \right)} \right| - \sum\nolimits_{j = 0}^{\nu  - 1} {\left| {p_{01}^j\left( {{\xi _ + }} \right) - p_{01}^j\left( \xi  \right)} \right|} \notag \\
		\geqslant{}& {\varpi _2}\left( {\left| {{\xi _ + } - \xi } \right|} \right) - c\left( {\sum\nolimits_{j = 0}^{\nu  - 1} {{\mu _j}} } \right){\varpi _1}\left( {\left| {{\xi _ + } - \xi } \right|} \right)\notag \\
		\geqslant{}& {\varpi _2}\left( {\left| {{\xi _ + } - \xi } \right|} \right)/2,
	\end{align}
	where the first inequality follows from \eqref{lxm}, and the second uses (\textbf{A2}), while the last holds since $ \varepsilon  $ is small enough.
	
	Therefore, we have
	\begin{equation}\label{kesaicha}
		\left| {{\xi _ + } - \xi } \right| \leqslant \varpi _2^{ - 1}\left( {2\left| {p_{01}^\nu \left( {{\xi _ + }} \right)} \right|} \right) \leqslant \varpi _2^{ - 1}\left( {2c{\mu _\nu }} \right) \leqslant c\varpi_1^{ - 1}\left( {2c{\mu _\nu }} \right) \leqslant c{\mu _\nu }.
	\end{equation}
	To be clear, the first inequality in \eqref{kesaicha} follows from \eqref{4.27}, the second uses \eqref{lxm}, and (\textbf{A2}) has been employed in the third, while the last is valid via $ \mathop {\overline {\lim } }\limits_{x \to {0^ + }} x/{{\varpi}_1 }\left( x \right) <  + \infty  $, see Definition \ref{definition1}. Recall the definition of $ \mu_\nu $ in \eqref{miuniu}, we therefore prove that $ {\left\{ {{\xi _j}} \right\}_{j \in {\mathbb{N}^ + }}} $ is a Cauchy sequence since the convergence of $ \mu_\nu $ is at least super-exponential.
	
	According to \eqref{kesaicha}, we eventually arrive at
	\begin{equation}\notag
		\left| {{\xi _ + } - {\xi _0}} \right| \leqslant \sum\nolimits_{j = 1}^{\nu  + 1} {\left| {{\xi _j} - {\xi _{j - 1}}} \right|}  \leqslant c\sum\nolimits_{j = 1}^\nu  {{\mu _j}}.
	\end{equation}
	From $ \xi  \in {B\left( {{\xi _0},\delta } \right)} $ in \eqref{4.25} and the fact $ \varepsilon  $ is small enough, we have $ {B_{c\mu }}\left( \xi  \right) \subset {B\left( {{\xi _0},\delta } \right)} $.
	
	\subsubsection{Estimate on $ {N_ + } $}
	\begin{lemma}
		There is a constant $ c_3>0 $ such that the following holds:
		\begin{equation}\notag
			\left| {{\xi _ + } - \xi } \right| \leqslant {c_3}\mu ,\;\;\left| {{e_ + } - e} \right| \leqslant {c_3}{s^4}\mu ,\;\;{\left\| {{{\bar h }_ + } - \bar h } \right\|_{D\left( s \right)}} \leqslant {c_3}{s^4}\mu .
		\end{equation}
	\end{lemma}
	\begin{proof}
		See Lemma 3.5 in \cite{R-18}.
	\end{proof}

	\subsubsection{Estimate on $ {\Phi _ + } $}
	\begin{lemma}
		There is a constant $ c_4>0 $ such that for all $ \left| i \right| + \left| j \right| \leqslant 4 $,
		\begin{equation}\notag
			{\left\| {\partial _x^i\partial _y^jF} \right\|_{\hat D}} \leqslant {c_4}\gamma _0^{4}{s^{4 - \left| i \right|}}\mu \Gamma \left( {r - {r_ + }} \right).
		\end{equation}
	\end{lemma}
	\begin{proof}
		See Lemma 3.6 in \cite{R-18}.
	\end{proof}

	\begin{lemma}
		Assume that
		\[{c_4}{s^{ 3}}\mu \Gamma \left( {r - {r_ + }} \right) < \left( {r - {r_ + }} \right)/8,\;\;{c_4}{s^4}\mu \Gamma \left( {r - {r_ + }} \right) < \alpha s/8.\]
		Then the following holds. For all $ 0 \leqslant t \leqslant 1 $, the maps $ \phi _F^t:{D_{\alpha /4}} \to {D_{\alpha /2}} $ and $ \phi :\mathcal{O} \to {\mathcal{O}_ + } $ are well defined, and $ {\Phi _ + }:{D_ + } \to D\left( {s,r} \right) $. There is a constant $ c_5>0 $ such that
		\begin{align*}
			{\left\| {\phi _F^t - id} \right\|_{\tilde D}},\;{\left\| {D\phi _F^t - Id} \right\|_{\tilde D}},\;{\left\| {{D^2}\phi _F^t} \right\|_{\tilde D}},&\\
			{\left\| {{\Phi _ + } - id} \right\|_{\tilde D}},\;{\left\| {D{\Phi _ + } - Id} \right\|_{\tilde D}},\;{\left\| {{D^2}{\Phi _ + }} \right\|_{\tilde D}}& \leqslant {c_5}\mu \Gamma \left( {r - {r_ + }} \right).
		\end{align*}
	\end{lemma}
	\begin{proof}
		See Lemma 3.7 in \cite{R-18}.
	\end{proof}

	\subsubsection{Estimate on $ {P _ + } $}
	\begin{lemma}
		Assume the previous assumptions hold. Then there is a constant $ c_6>0 $ such that
		\begin{equation}\notag
			{\left\| {{P_ + }} \right\|_{{D_ + }}} \leqslant {c_6}\gamma _0^{5}{s^4}{\mu ^2}\left( {{\Gamma ^2}\left( {r - {r_ + }} \right) + \Gamma \left( {r - {r_ + }} \right)} \right).
		\end{equation}
		Moreover, if
		\begin{equation}\notag
			{\mu ^\rho }\left( {{\Gamma ^2}\left( {r - {r_ + }} \right) + \Gamma \left( {r - {r_ + }} \right)} \right) \leqslant 1,
		\end{equation}
		then
		\begin{equation}\notag
			{\left\| {{P_ + }} \right\|_{{D_ + }}} \leqslant {c_6}\gamma _0^{5}s_ + ^4{\mu _ + }.
		\end{equation}
	\end{lemma}
	
	\section{Proof of Theorem \ref{theorem1}}\label{sec-5}
	\subsection{Iteration lemma}
	In this subsection, we will give an iteration lemma which guarantees the inductive construction of the transformations in all KAM steps.
	
	Let $ {r_0},{s_0},{\gamma _0},{\mu _0},{H_0},{e_0},{{\bar h}_0},{P_0} $ be given at the beginning of Section \ref{sec-4} and let $ {D_0} = {D_0}\left( {{s_0},{r_0}} \right),{K_0} = 0,{\Phi _0} = id $. We define the following sequence inductively for all $ \nu  \geqslant 1 $
	\begin{align}
		{r_\nu } ={}& {r_0}\big( {1 - \sum\nolimits_{i = 1}^\nu  {{2^{ - i - 1}}} } \big),\notag \\
		{s_\nu } ={}& {\alpha _{\nu  - 1}}{s_{\nu  - 1}}/8,\notag \\
		{\alpha _\nu } ={}& \mu _\nu ^{2\rho } = \mu _\nu ^{1/5},\notag \\
		{\mu _\nu } ={}& {8^4}{c_0}\mu _{\nu  - 1}^{1 + \rho },\notag \\
		{K_\nu } ={}& {\left( {\left[ { - \ln {\mu _{\nu  - 1}}} \right] + 1} \right)^{3\eta }},\notag \\
		{{\tilde D}_\nu } ={}& D\left( {{s_\nu }/2,{r_\nu } + 6\left( {{r_{\nu  - 1}} - {r_\nu }} \right)/8} \right).\notag
	\end{align}
	\begin{lemma}\label{Iteration lemma}
		Denote $ {\mu _ * } = {\mu _0}/\big( {{{\left( {{M^ * } + 2} \right)}^{3}}K_1^{5\left( {\tau  + 1} \right)}} \big) $. If $ \varepsilon  $ is small enough, then the KAM step described on the above is valid for all $ \nu  \geqslant 0 $, resulting the sequences $ {H_\nu },{N_\nu },{e_\nu },{{\bar h}_\nu },{P_\nu },{\Phi _\nu } $ for $ \nu  \geqslant 1 $ with the following properties:
		\begin{align*}
			{}&\left| {{e_{\nu  + 1}} - {e_\nu }} \right|,\;{\left\| {{{\bar h}_{\nu  + 1}} - {{\bar h}_\nu }} \right\|_{D\left( {{s_\nu }} \right)}},\;{\left\| {{P_\nu }} \right\|_{D\left( {{s_\nu },{r_\nu }} \right)}},\;\left| {{\xi _{\nu  + 1}} - {\xi _\nu }} \right| \leqslant \mu _ * ^{1/2}{2^{ - \nu }},\\
			{}&\left| {{e_\nu } - {e_0}} \right|,\;{\left\| {{{\bar h}_\nu } - {{\bar h}_0}} \right\|_{D\left( {{s_\nu }} \right)}} \leqslant 2\mu _ * ^{1/2}.
		\end{align*}
		In addition, $ {\Phi _{\nu  + 1}}:{{\tilde D}_{\nu  + 1}} \to {{\tilde D}_\nu } $ is symplectic, and
		\begin{equation}\label{dafai}
			{\left\| {{\Phi _{\nu  + 1}} - id} \right\|_{{{\tilde D}_{\nu  + 1}}}} \leqslant \mu _ * ^{1/2}{2^{ - \nu }}.
		\end{equation}
		Moreover, on $ {D_{\nu  + 1}} $,
		\[{H_{\nu  + 1}} = {H_\nu } \circ {\Phi _{\nu  + 1}} = {N_{\nu  + 1}} + {P_{\nu  + 1}}.\]
	\end{lemma}
	\begin{proof}
		See Lemma 3.8 in \cite{R-18}.
	\end{proof}

	\subsection{Convergence}
	The convergence is standard. For the sake of completeness, we briefly give the framework of proof. Let
	\[{\Psi ^\nu }: = {\Phi _1} \circ {\Phi _2} \circ  \cdots {\Phi _\nu },\;\;\nu  \geqslant 1.\]
	By Lemma \ref{Iteration lemma}, we have
	\[{D_{\nu  + 1}} \subset {D_\nu },\;\;{\Psi ^\nu }:{{\tilde D}_\nu } \to {{\tilde D}_0},\;\;{H_0} \circ {\Psi ^\nu } = {H_\nu } = {N_\nu } + {P_\nu },\]
	and
	\[{N_ + } = {e_\nu } + \Big\langle {\omega \left( {{\xi _\nu }} \right) + \sum\nolimits_{j = 0}^\nu  {p_{01}^j\left( {{\xi _\nu }} \right)} ,y} \Big\rangle  + {{\bar h}_\nu }\left( {y,{\xi _\nu }} \right),\;\;\nu  \geqslant 0,\]
	where $ {\Psi ^0} = id $. Using \eqref{dafai} and the identity
	\[{\Psi ^\nu } = id + \sum\nolimits_{j = 0}^\nu  {\left( {{\Psi ^j} - {\Psi ^{j - 1}}} \right)} ,\]
	it is easy to verify that $ {\Psi ^\nu } $ is uniformly convergent and denote the limit by $ {\Psi ^\infty } $.
	
	In view of Lemma \ref{Iteration lemma}, it is obvious to see that $ {e_\nu },{{\bar h}_\nu },{\xi _\nu } $ converge uniformly about $ \nu $, and denote their limits by $ {e_\infty },{{\bar h}_\infty },\xi _\infty:=\xi^* $, respectively. By Lemma \ref{crucial}, we have
	\begin{align}
		{}&\omega \left( {{\xi _1}} \right) + p_{01}^0\left( {{\xi _1}} \right) = \omega \left( {{\xi _0}} \right),\notag \\
		{}&\omega \left( {{\xi _2}} \right) + p_{01}^0\left( {{\xi _2}} \right) + p_{01}^1\left( {{\xi _2}} \right) = \omega \left( {{\xi _0}} \right),\notag \\
		{}&\vdots\notag \\
		\label{pinlvlie}{}&\omega \left( {{\xi _\nu }} \right) + p_{01}^0\left( {{\xi _\nu }} \right) +  \cdots  + p_{01}^{\nu  - 1}\left( {{\xi _\nu }} \right) = \omega \left( {{\xi _0}} \right).
	\end{align}
	Taking limits at both sides of \eqref{pinlvlie}, we get
	\[\omega \left( {{\xi _\infty }} \right) + \sum\nolimits_{j = 0}^\infty  {p_{01}^j\left( {{\xi _\infty }} \right)}  = \omega \left( {{\xi _0}} \right).\]
	
	Then, on $ D\left( {{s_0}/2} \right) $, $ N_\nu $ converges uniformly to
	\[{N_\infty } = {e_\infty } + \left\langle {\omega \left( {{\xi _0}} \right),y} \right\rangle  + {{\bar h}_\infty }\left( {y,{\xi _\infty }} \right).\]
	Hence, on $ D\left( {{s_0}/2,{r_0}/2} \right) $,
	\[{P_\nu } = {H_0} \circ {\Psi ^\nu } - {N_\nu }\]
	converges uniformly to
	\[{P_\infty } = {H_0} \circ {\Psi ^\infty } - {N_\infty }.\]
	Since
	\[{\left\| {{P_\nu }} \right\|_{{D_\nu }}} \leqslant c\gamma _0^{5}s_\nu ^4{\mu _\nu },\]
	we have that it converges to $ 0 $ as $ \nu  \to \infty  $. So, on $ D\left( {0,{r_0}/2} \right) $,
	\[J\nabla {P_\infty } = 0.\]
	Thus, for the given $ {\xi _0} \in \mathcal{O} $, the Hamiltonian
	\[{H_\infty } = {N_\infty } + {P_\infty }\]
	admits an analytic, quasi-periodic, invariant $ n $-torus $ {\mathbb{T}^n} \times \left\{ 0 \right\} $ with the Diophantine frequency $ \omega \left( {{\xi _0}} \right) $, which is the corresponding unperturbed toral frequency.
	
	This completes the proof of Theorem \ref{theorem1}.
\end{proof}

\section{Further comments and the infinite dimensional version}\label{sec-7}

\subsection{Further comments}\label{comments}
\begin{itemize}
	%
	%
\item[(1)] Easy to see that the 
condition (\textbf{A1}) cannot be removed in the sense of solvability of the frequency equation. However, it could be weaken to the following:

(\textbf{A1*})\ Let $ \tt{p}=\omega(\xi_0)\in {\tilde \Omega } $ satisfy the Diophantine condition, where $ {\tilde \Omega } $ is an open set of $ \omega(\Omega) $, and $\Omega\subset{\mathcal{O}}$ is open. 

		At this point we do not need the Brouwer degree, and thus remove the limitation on the dimension of the parameter $ \xi $. Note that $ {\omega ^{ - 1}}( {\tilde \Omega } ) $ is also an open set because $ \omega $ is continuous.  Therefore, as long as the perturbation in KAM is sufficiently small (the smallness may depend on $ \tt{p} $, i.e., the position of $ {\omega ^{ - 1}}\left( \tt{p} \right) $),  the solvability of the frequency equation does not change thanks to the continuity of $ \omega $ (note that we avoid the boundary of range), and the uniform convergence of $\{\xi_\nu\}$ can still be proven by Cauchy theorem.
	
	\item[(2)] It should be pointed out that (\textbf{A2}) does not imply (\textbf{A1}) (or (\textbf{A1*})), and it is a somewhat strong condition. There does not exist such functions  in the case of $ n=1 $ if $ \mathop {\lim }\limits_{x \to {0^ + }} {\varpi _2}\left( x \right)/x =  + \infty  $ (for instance, $ {\varpi _2}\left( x \right) = \sqrt x $, i.e., H\"older's type), since $ \omega $ must be nowhere monotonic, but $ \varpi_2(x) $ could be $ x $ (Lipschitz's type).  The above argument does not hold if $ n \geqslant 2 $.
	
	\item[(3)] The semi norm of modulus of continuity is not necessary since we only need the uniform convergence throughout this paper, and the new perturbation at each step keeps $ \varpi_1 $ continuous because the transformation is analytic.

	\item[(4)] We claim that the conclusion of Theorem \ref{theorem1} even holds on a dense subset $ \mathcal{A} \subseteq  \mathcal{O} $ of parameter of zero Lebesgue measure in our approach, such as $ \mathcal{A}=\mathcal{O} \cap \mathbb{Q}^n  $, that is, \textbf{it characterizes the dynamics on the set of zero Lebesgue measure.} However, constraints outside $ \mathcal{A} $ (i.e., on the set $ B\left( {{\xi _0},\delta } \right) $), such as (\textbf{A1}) (or (\textbf{A1*})) and (\textbf{A2}), are indispensable. There are at least two approaches to prove the frequency-preserving KAM persistence, namely  continuation method and perturbation method. For the former, in view of the density, all the related functions can be continuously extended from $ \mathcal{A} $ to $ \mathcal{O} $ keeping all properties, we therefore directly obtain the desired conclusion. 	As to the latter, we show the perturbation technique for the key step, i.e., the Cauchy uniform convergence of $ \left\{ {{\vartheta _\nu }} \right\}_{\nu  = 1}^\infty  $. 	At this point, $ \xi_\nu \in B\left( {{\xi _0},\delta } \right) $ in \eqref{omega} may not be in the subset $ \mathcal{A} $. Fortunately, we could choose $ \vartheta_\nu \in \mathcal{A} $ such that $ \left| {{\xi _\nu } - {\vartheta _\nu }} \right| $ is sufficiently small for all $ \nu \in \mathbb{N}^+ $ in view of the density, and our iteration will not be affected. We therefore obtain that
	\begin{align}
		{}&\omega \left( {{\vartheta _1}} \right) + p_{01}^0\left( {{\vartheta _1}} \right) = \omega \left( {{\xi _0}} \right)+{\kappa _{1} },\notag \\
		{}&\omega \left( {{\vartheta_2}} \right) + p_{01}^0\left( {{\vartheta _2}} \right) + p_{01}^1\left( {{\vartheta _2}} \right) = \omega \left( {{\xi _0}} \right)+{\kappa _{2}},\notag \\
		{}&\vdots\notag \\
		\label{A1pinlv}{}&\omega \left( {{\vartheta _\nu }} \right) + p_{01}^0\left( {{\vartheta _\nu }} \right) +  \cdots  + p_{01}^{\nu  - 1}\left( {{\vartheta _\nu }} \right) = \omega \left( {{\xi _0}} \right)+{\kappa _{\nu}},
	\end{align}
	i.e., we allow a slight perturbation of the frequency at this point, but it must be of Diophantine type (full measure in $ \mathbb{R}^n $). In view of the density and the continuity of $ \omega $ and $ P $, we have $ \mathop {\lim }\limits_{\nu  \to \infty } {\kappa _\nu } = 0 $, and this leads to $ \mathop {\lim }\limits_{\nu  \to \infty } \left( {\omega \left( {{\xi _0}} \right) + {\kappa _\nu }} \right) = \omega \left( {{\xi _0}} \right) $, i.e., the prescribed frequency keeps unchanged. The uniform convergence of $ \left\{ {{\vartheta _\nu }} \right\}_{\nu  = 1}^\infty  $ can similarly be proved as that in \eqref{4.27}, due to the density of $ \mathcal{A} $. Namely, we could choose $ \vartheta_\nu \in \mathcal{A} $ such that
	\[\left\{ \begin{gathered}
		\left| {p_{01}^\nu \left( {{\vartheta _{\nu  + 1}}} \right)} \right| \leqslant 2\left| {p_{01}^\nu \left( {{\xi _{\nu  + 1}}} \right)} \right|, \hfill \\
		\left| {{\kappa _\nu }} \right| + \left| {{\kappa _{\nu  + 1}}} \right| \leqslant \frac{1}{4}{\varpi _2}\left( {\left| {{\xi _{\nu  + 1}} - {\xi _\nu }} \right|} \right), \hfill \\
		{\varpi _2}\left( {\left| {{\xi _{\nu  + 1}} - {\xi _\nu }} \right|} \right) \leqslant 2{\varpi _2}\left( {\left| {{\vartheta _{\nu  + 1}} - {\vartheta _\nu }} \right|} \right), \hfill \\
		\left| {p_{01}^j\left( {{\vartheta _{\nu  + 1}}} \right) - p_{01}^j\left( {{\vartheta _\nu }} \right)} \right| \leqslant 2\left| {p_{01}^j\left( {{\xi _{\nu  + 1}}} \right) - p_{01}^j\left( {{\xi _\nu }} \right)} \right|,\;\;0 \leqslant j \leqslant \nu  - 1. \hfill \\
	\end{gathered}  \right.\]
	Therefore, by \eqref{4.27} and \eqref{A1pinlv} we have
	\begin{align*}
		&2\left| {p_{01}^\nu \left( {{\xi _{\nu  + 1}}} \right)} \right| \geqslant \left| {p_{01}^\nu \left( {{\vartheta _{\nu  + 1}}} \right)} \right|\\
		\geqslant &\left| {\omega \left( {{\vartheta _{\nu  + 1}}} \right) - \omega \left( {{\vartheta _\nu }} \right)} \right| - \sum\limits_{j = 0}^{\nu  - 1} {\left| {p_{01}^j\left( {{\vartheta _{\nu  + 1}}} \right) - p_{01}^j\left( {{\vartheta _\nu }} \right)} \right|}  - \left| {{\kappa _\nu }} \right| - \left| {{\kappa _{\nu  + 1}}} \right|\\
		\geqslant &{\varpi _2}\left( {\left| {{\vartheta _{\nu  + 1}} - {\vartheta _\nu }} \right|} \right) - 2\sum\limits_{j = 0}^{\nu  - 1} {\left| {p_{01}^j\left( {{\xi _{\nu  + 1}}} \right) - p_{01}^j\left( {{\xi _\nu }} \right)} \right|}  - \frac{1}{4}{\varpi _2}\left( {\left| {{\xi _{\nu  + 1}} - {\xi _\nu }} \right|} \right)\\
		\geqslant &{\varpi _2}\left( {\left| {{\vartheta _{\nu  + 1}} - {\vartheta _\nu }} \right|} \right) - 2c\left( {\sum\limits_{j = 0}^{\nu  - 1} {{\mu _j}} } \right){\tilde{\varpi} _1}\left( {\left| {{\xi _{\nu  + 1}} - {\xi _\nu }} \right|} \right) - \frac{1}{2}{\varpi _2}\left( {\left| {{\vartheta _{\nu  + 1}} - {\vartheta _\nu }} \right|} \right)\\
		\geqslant &\frac{1}{2}{\varpi _2}\left( {\left| {{\vartheta _{\nu  + 1}} - {\vartheta _\nu }} \right|} \right) - \frac{1}{4}{\tilde{\varpi} _1}\left( {\left| {{\vartheta _{\nu  + 1}} - {\vartheta _\nu }} \right|} \right)\\
		\geqslant &\frac{1}{2}{\varpi _2}\left( {\left| {{\vartheta _{\nu  + 1}} - {\vartheta _\nu }} \right|} \right) - \frac{1}{4}{\varpi _2}\left( {\left| {{\vartheta _{\nu  + 1}} - {\vartheta _\nu }} \right|} \right)\\
		\ = &\frac{1}{4}{\varpi _2}\left( {\left| {{\vartheta _{\nu  + 1}} - {\vartheta _\nu }} \right|} \right).
	\end{align*}
	This leads to
	\[\left| {{\vartheta _{\nu  + 1}} - {\vartheta _\nu }} \right| \leqslant \varpi _2^{ - 1}\left( {8\left| {p_{01}^\nu \left( {{\xi _{\nu  + 1}}} \right)} \right|} \right) \leqslant c\left| {p_{01}^\nu \left( {{\xi _{\nu  + 1}}} \right)} \right| \leqslant c{\mu _\nu },\]
	which implies that $ \left\{ {{\vartheta _\nu }} \right\}_{\nu  = 1}^\infty  $ is a Cauchy sequence since the convergence rate of  $ {\mu _\nu } $ is at least super-exponential as we forego. This gives the claim.

	\item[(5)] It can be seen from \eqref{4.27} that we actually only use the weak regularity of $ P $ near $ \xi_0=\omega^{-1}(\tt{p}) $ (i.e.,  the local neighborhood in (\textbf{A2}), see Remark \ref{Remarkk3.2}). Therefore, for some given $ \tt{p} $, as long as the perturbation $ P $ has good regularity near $ \omega^{-1}(\tt{p}) $, our conclusion still holds, no matter how weak the regularity of $ P $ is elsewhere (obviously the regularity of $ P $ is very weak on the entire set $ \mathcal{O} $ at this point). An explicit example is constructed in subsection \ref{xianshi}.
	
	The above analysis holds even under the perspective of comment (4).

\end{itemize}

\subsection{Frequency-preserving KAM under weak regularity in infinite dimensional case}\label{wuqiongw}

Here we consider the infinite dimensional case. The spatial structure together with the corresponding weighted norms and the nonresonant condition of frequency we need were introduced by P\"{o}schel \cite{MR1037110}.

Let $ \Lambda $ be an infinite dimensional lattice with a weighted spatial structure $ \mathcal{S} $, where $ \mathcal{S} $ is a family of finite subsets $ A $ of $ \Lambda $. Namely, $ \mathcal{S} $ is a spatial structure on $ \Lambda $ characterized by the property that the union of any two sets in $ \mathcal{S} $ is again in $ \mathcal{S} $, if they intersect:
\[A,B \in \mathcal{S},\;\;A \cap B \ne \phi  \Rightarrow A\cup B  \in \mathcal{S}.\]
Then we introduce a nonnegative weight function $ \left[  \cdot  \right]:A \to \left[ A \right] $ defined on $ \mathcal{S}\cap \mathcal{S}  = \left\{ {A\cap B :A,B \in \mathcal{S}} \right\} $ to reflect the size, location and something else of the set $ A $. The weight function satisfies the monotonicity and subadditivity for all $ A,B $ in $ \mathcal{S} $
\begin{align*}
	A \subseteq B &\Rightarrow \left[ A \right] \leqslant \left[ B \right]\\
	A \cap B \ne \phi  &\Rightarrow \left[ {A \cup B} \right] + \left[ {A \cap B} \right] \leqslant \left[ A \right] + \left[ B \right].
\end{align*}
Next we define the norms for $ k $ runs over all nonzero
integer vectors in $ \mathbb{Z}^\Lambda $ whose support $ \operatorname{supp} k = \left\{ {\lambda :{k_\lambda } \ne 0} \right\} $ is a finite set:
\[\left| k \right|: = \sum\limits_{\lambda  \in \Lambda } {\left| {{k_\lambda }} \right|} ,\;\;\left[ {\left[ k \right]} \right] = \mathop {\min }\limits_{\operatorname{supp} k \subseteq A \in \mathcal{S}} \left[ A \right].\]
At this point, the infinite dimensional nonresonant condition can be defined as follows.

\begin{definition}[Infinite dimensional nonresonant condition]\label{ifnc}
	Given a nondecreasing approximation function $ \Delta :\left[ {0, + \infty } \right) \to \left[ {1, + \infty } \right) $, that is,  $ \Delta \left( 0 \right) = 1 $, and
	\[\frac{{\log \Delta \left( t \right)}}{t} \searrow 0,\;\;0 \leqslant t \to  + \infty, \]
	and
	\[\int_1^{ + \infty } {\frac{{\log \Delta \left( t \right)}}{t}dt}  <  + \infty .\]
	Then for some $ \alpha>0 $ and every $ 0 \ne k \in {\mathbb{Z}^\Lambda } $ with finite support, the infinite dimensional nonresonant condition read
	\[\left| {\left\langle {k,\omega } \right\rangle } \right| \geqslant \frac{\alpha }{{\Delta \left( {\left[ {\left[ k \right]} \right]} \right)\Delta \left( {\left| k \right|} \right)}}.\]
	
\end{definition}

Let $ N =  \left\langle {\omega ,I} \right\rangle  $ be the unperturbed integrable Hamiltonian with $ \omega \left( \xi  \right):{\mathbb{R}^\Lambda } \supseteq \mathcal{B}  \to {\mathbb{R}^\Lambda } $, and $ P $ be the perturbation of the form $ P = \sum\limits_{A \in \mathcal{S}} {{P_A}\left( {{\varphi _A},{I_A};{\xi_A }} \right)}  $, where $ {\varphi _A} = \left( {{\varphi _\lambda }:\lambda  \in A} \right) $, and similarly $ I_A $
and $ \xi_A $. Suppose that the perturbed Hamiltonian

\begin{equation}\label{infinteH}
H = N + P =  \left\langle {\omega ,I} \right\rangle  + \sum\limits_{A \in S} {{P_A}\left( {{\varphi _A},{I_A};{\xi_A}} \right)}
\end{equation}
is real analytic in the phase space variables $ \phi, I $ on a complex neighbourhood
\[{\mathcal{D}_{r,s}}:{\left| {\operatorname{Im} \varphi } \right|_\infty } < r,\;\;{\left| I \right|_w} < s,\]
of the torus $ {\mathcal{T}_0}: = {\mathbb{T}^\Lambda } \times \left\{ 0 \right\} $, and continuous with respect to the parameter $ \xi $ on $ \mathcal{B} $, where $ w>0 $. The norms are
\[{\left| \varphi  \right|_\infty } = \mathop {\sup }\limits_{\lambda  \in \Lambda } \left| {{\varphi _\lambda }} \right|,\;\;{\left| I \right|_w} = \sum\limits_{\lambda  \in \Lambda } {\left| {{I_\lambda }} \right|{e^{w\left[ \lambda  \right]}}} ,\;\;\left[ \lambda  \right] = \mathop {\min }\limits_{\lambda  \in A \in \mathcal{S} \cap \mathcal{S}} \left[ A \right].\]
The size of the perturbation is measured in terms of the
weighted norm
\[|||P|||{_{m,r,s}}: = \sum\limits_{A \in S} {{{\left\| {{P_A}} \right\|}_{r,s}}{e^{m\left[ A \right]}}} ,\]
where $ {P_A} = \sum\limits_k {{P_{A,k}}\left( {I,\xi } \right){e^{i\left\langle {k,\varphi } \right\rangle }}}  $ is the Fourier series expansion, and
\[\|{P_A}\|{_{r,s}}: = \sum\limits_{k \in {\mathbb{Z}^\Lambda }} {{{\left\| {{P_{A,k}}} \right\|}_s}{e^{r\left| k \right|}}} ,\;\;{\left\| {{P_{A,k}}} \right\|_s}: = \mathop {\sup }\limits_{{{\left| I \right|}_w} < s,\;\xi  \in \mathcal{B}} \left| {{P_{A,k}}\left( {I,\xi } \right)} \right|.\]
Since quantitative estimates of the smallness of the  perturbation are not emphasized here, we omit some notations in \cite{MR1037110}.

Similar to Theorem \ref{theorem1} of the finite dimensional case, we make the following assumptions in view of the comment (1) in Subsection \ref{comments}:

(\textbf{B1}) Let $ \tt{q}$ $\in {\tilde{ \mathcal{B}} } $, where $ {\tilde{ \mathcal{B}} } $ is an open set of $ \omega(\mathcal{B}) $. Assume that $ \tt{q} $ satisfies the infinite dimensional nonresonant condition, see Definition \ref{ifnc}.

(\textbf{B2}) Assume that the perturbation $ P $ is $ {\varpi}_1 $ continuous on $ \tilde{ \mathcal{B}}  $, and there exists a modulus of continuity $ \varpi_2 $ 
such that $ \varpi_1 \lesssim \varpi_2 $, 
and
\begin{equation}\notag
	\left| {\omega \left( \xi  \right) - \omega \left( \zeta  \right)} \right| \geqslant {\varpi _2}\left( {\left| {\xi  - \zeta } \right|} \right),\;\;\forall 0 < \left| {\xi  - \zeta } \right| \leqslant {1 },\;\;\xi ,\zeta  \in \tilde{ \mathcal{B}} .
\end{equation}

Here we show why the frequency-preserving technique in this paper could be applied to infinite dimensional KAM theory. Recall that Subsubsection \ref{kkkk} uses Brouwer degree in the finite dimensional case, but this is actually not necessary, as mentioned in comment (1) in Subsection \ref{comments}, only the openness and continuity are enough to ensure the solvability of the frequency equation and thus frequency-preserving KAM persistence is admited.  Crucially, (\textbf{B1}) is not limited by the dimension of the parameter and the frequency mapping, and therefore our frequency-preserving technique is not affected. Besides, the nonresonance condition is adapted to the classical KAM iteration in \cite{MR1037110} and also has no effect on our proof.  The framework for infinite dimensional KAM is the same as that in Sections \ref{sec-4} and \ref{sec-5}, that is, instead of digging out domain, we use the technique of parameter translation to keep the prescribed frequency unchanged, 
consequently, the measure estimates are not involved in our KAM approach. More preciously, denote $ {{\tilde \xi }_0}: = {\omega ^{ - 1}}\left( {\tt{q}} \right) \in \mathcal{B}  \subseteq {\mathbb{R}^\Lambda } $, then by (\textbf{B1}) we could find a sequence $ {\{ {{{\tilde \xi }_\nu }} \}_\nu } $ near $ {{\tilde \xi }_0} $ due to the KAM smallness through iteration and the continuity of $ \omega $, such that $ \omega(\tilde \xi_\nu)=\tt{q} $, i.e., is of frequency-preserving. Except for this, $ {\{ {{{\tilde \xi }_\nu }} \}_\nu } $ is indeed a  Cauchy sequence by applying (\textbf{B2}), similar to \eqref{kesaicha}. Finally, the uniform convergence of KAM iteration could be obtained directly. We therefore give the following infinite dimensional version without proof.


\begin{theorem}
	Consider Hamiltonian systems \eqref{infinteH}. Assume that  (\textbf{B1}) and (\textbf{B2}) hold. Then there exists $ \tilde\xi^* \in \tilde{ \mathcal{B}}  $ as long as $ |||P|||{_{m,r,s}} $ is sufficiently small, such that the perturbed Hamiltonian system $ H( {I,\varphi ,\tilde\xi^* } ) $ admits an invariant torus with infinite dimensional nonresonant frequency $ \tt{q}=\omega(\tilde\xi^*) $.
\end{theorem}
\begin{remark}
	Comments similar to that in Subsection \ref{comments} can also be made, we do not pursue this point.
\end{remark}

\section{Applications}\label{sec-6}
\subsection{KAM via arbitrary H\"{o}lder's type of the perturbation}
As a direct application of Theorem \ref{theorem1}, we give the following KAM theorem that the perturbation $ P $ is arbitrary H\"{o}lder continuous with respect to parameter $ \xi $, which is the first result to our knowledge.
\begin{theorem}\label{theorem2}
	Consider Hamiltonian system \eqref{H}, where the perturbation $ P $ is H\"{o}lder continuous about the parameter $ \xi $. Assume that (\textbf{A1}), (\textbf{A2}) 
	hold, where $ \varpi_2 $ in (\textbf{A2}) is Logarithmic Lipschitz, see \ref{logarithmic Lipschitz}. Then there exists a sufficiently small $ {\varepsilon _0} > 0 $, for any $ 0 < \varepsilon  < {\varepsilon _0} $, there exists $ {\xi^* } \in \mathcal{O} $, such that the perturbed Hamiltonian system $ H\left( {y,x,{\xi^* },\varepsilon } \right) $ admits an invariant torus with frequency $\tt{p}= \omega \left( {{\xi _0}} \right) $.	
\end{theorem}
\begin{proof}
	It only needs to be noted that according to Definition \ref{qr}, H\"{o}lder's $ \varpi_1 $ is not weaker than Logarithmic Lipschitz's $ \varpi_2 $, thus Theorem \ref{theorem1} is applied to complete the proof.
\end{proof}

\subsection{An explicit example under weak regularity on a set of zero Lebesgue measure}\label{xianshi}


From the point of view of Baire category, the regularity of the majority of functions is very weak, such as nowhere differentiable (Brownian motion, etc.). Therefore, traditional KAM theorems have somewhat limitations on requirements of  parameters (Lipschitz's type). In fact, the regularity of nowhere differentiable can be even worse, for example, given a modulus of continuity $ \varpi_1 $, there exists a family of continuous functions that are nowhere $ \varpi_1 $ continuous. Namely, we have the following theorem.

Let  $ {\left\{ {{b_n}} \right\}_{n \in {\mathbb{N}^ + }}} $ be a positive sequence, satisfies $ \mathop {\lim }\limits_{n \to \infty } {b_n} = 0 $ and $ {2^{ - 1}}{b_m} - \sum\limits_{n = m + 1}^\infty  {{b_n}}  > 0 $ for every $ n \geqslant m + 1 \in \mathbb{N} $. Note that $ \sum\limits_{n = 1}^\infty  {{b_n}}  <  + \infty  $ at this point, and the convergence rate is very fast (at least exponentially), as an illustration, let $ {b_n} = {e^{ - c{n^2}}} $ with some $ c>0 $ sufficiently large. Then we show that:

\begin{theorem}\label{nonowh}
	Given a modulus of continuity $ {\varpi _1} $, there exists a  function (actually, a family)
	\[f\left( x \right) = \sum\limits_{n = 1}^\infty  {{b_n}\sin \left( {\pi {a_n}x} \right)}\]
	on $ \mathbb{R} $ and a modulus of continuity $ {\varpi _2} \gtrsim  {\varpi _1}$, such that $ f $ is $ {\varpi _2} $ continuous, but nowhere $ {\varpi _1} $ continuous.
	
\end{theorem}
\begin{remark}\label{noholder}
	As a direct application, we can construct a family of functions, which are nowhere H\"older continuous.
\end{remark}

\begin{proof}
	Here we construct a positive integer sequence $ {\left\{ {{a_n}} \right\}_{n \in {\mathbb{N}^ + }}} $, such that
	\begin{itemize}
		{\item[(a1)] $ {\left\{ {{a_n}a_{n - 1}^{ - 1},{a_n}} \right\}_{n \in {\mathbb{N}^ + }}} \in \left\{ {2k:k \in {\mathbb{N}^ + }} \right\} $;
		\item[(a2)] $ {a_m}a_{m - 1}^{ - 1} \geqslant \left( {2\pi \sum\limits_{n = 1}^\infty  {{b_n}} } \right)b_m^{ - 1} $ for $ m \in {\mathbb{N}^ + } $;
		\item[(a3)] $ {a_m} \geqslant \left( {\varpi _1^{ - 1}\left( {{m^{ - 1}}\left( {{2^{ - 1}}{b_m} - \sum\limits_{n = m + 1}^\infty  {{b_n}} } \right)} \right)} \right) $ for $ m \in {\mathbb{N}^ + } $, here $ {\varpi _1^{ - 1}} $ denotes the inverse of $ \varpi_1 $.}
\end{itemize}	
		It should be noted that the above three conditions are not harsh, as it can be seen later, (a1) is to make part of the tail of the difference vanish, while (a2) and (a3) only show that the growth rate of $ a_n $ is very fast, and they together ensure that the limit of the difference is infinite. One can construct many examples that satisfy these conditions.
		
		Notice that
		\[\sum\limits_{n = 1}^\infty  {\left| {{b_n}\sin \left( {\pi {a_n}x} \right)} \right|}  \leqslant \sum\limits_{n = 1}^\infty  {{b_n}}  <  + \infty ,\]
		then by applying the dominated convergence theorem we obtain the continuity of $ f $ on $ \mathbb{R} $. Since $ f $ is an odd periodic function, then it has a modulus of continuity $ \varpi_2 $, and we will prove that $ {\varpi _1} \lesssim {\varpi _2} $.
		
		Fix $ x>0 $, and let $ m \in \mathbb{N}^+ $ sufficiently large such that $ {a_m}x \geqslant 1 $. At this point we can write $ {a_m}x = {N_m} + {r_m} $, where $ {N_m} \in {\mathbb{N}^ + }$ and $0 < \left| {{r_m}} \right| \leqslant {2^{ - 1}} $. Take $ \Delta x = a_m^{ - 1}\left( { - \operatorname{sgn} \left( {{r_m}} \right){2^{ - 1}} - {r_m}} \right)  \to 0$ as $ m \to \infty $. This implies that $ {2^{ - 1}} \leqslant \left| {\Delta x} \right|{a_m} \leqslant 1 $. Then
		\begin{align*}
			&\frac{{f\left( {x + \Delta \left( x \right)} \right) - f\left( x \right)}}{{{\varpi _1}\left( {\left| {\Delta x} \right|} \right)}} \\
			= &\frac{{{b_m}}}{{{\varpi _1}\left( {\left| {\Delta x} \right|} \right)}}\left( {\sin \left( {\pi {a_m}\left( {x + \Delta x} \right)} \right) - \sin \left( {\pi {a_m}x} \right)} \right) \\
			&+ \sum\limits_{n = 1}^{m - 1} {\frac{{{b_n}}}{{{\varpi _1}\left( {\left| {\Delta x} \right|} \right)}}\left( {\sin \left( {\pi {a_n}\left( {x + \Delta x} \right)} \right) - \sin \left( {\pi {a_n}x} \right)} \right)} \\
			& + \sum\limits_{n = m + 1}^\infty  {\frac{{{b_n}}}{{{\varpi _1}\left( {\left| {\Delta x} \right|} \right)}}\left( {\sin \left( {\pi {a_n}\left( {x + \Delta x} \right)} \right) - \sin \left( {\pi {a_n}x} \right)} \right)} \\
			: = &{S_1} + {S_2} + {S_3}.
		\end{align*}
		
		Firstly, direct calculation gives
		\begin{align*}
			\left| {{S_1}} \right| &= \frac{{{b_m}}}{{{\varpi _1}\left( {\left| {\Delta x} \right|} \right)}}\left| {\sin \left( {\pi {a_m}\left( {x + \Delta x} \right)} \right) - \sin \left( {\pi {a_m}x} \right)} \right|\\
			& = \frac{{{b_m}}}{{{\varpi _1}\left( {\left| {\Delta x} \right|} \right)}}\left| {{{\left( { - 1} \right)}^{{N_m} + 1}}\left( {\operatorname{sgn} \left( {{r_m}} \right) + \sin \left( {\pi {r_m}} \right)} \right)} \right|\\
			& = \frac{{{b_m}}}{{{\varpi _1}\left( {\left| {\Delta x} \right|} \right)}}\left| {\operatorname{sgn} \left( {{r_m}} \right) + \sin \left( {\pi {r_m}} \right)} \right|\\
			& \geqslant \frac{{{b_m}}}{{{\varpi _1}\left( {a_m^{ - 1}} \right)}}.
		\end{align*}
		
		Secondly, recall the definition of modulus of continuity, we therefore assume that $ \frac{x}{{{\varpi _1}\left( x \right)}} $ is also a modulus of continuity without loss of generality. Then by applying the mean value theorem we obtain that
		\begin{align*}
			\left| {{S_2}} \right| &\leqslant \sum\limits_{n = 1}^{m - 1} {\frac{{{b_n}}}{{{\varpi _1}\left( {\left| {\Delta x} \right|} \right)}}\left| {\sin \left( {\pi {a_n}\left( {x + \Delta x} \right)} \right) - \sin \left( {\pi {a_n}x} \right)} \right|} \\
			&\leqslant \sum\limits_{n = 1}^{m - 1} {\frac{{{b_n}}}{{{\varpi _1}\left( {\left| {\Delta x} \right|} \right)}} \cdot \pi {a_n}\left| {\Delta x} \right|} \\
			& \leqslant \left( {\pi \sum\limits_{n = 1}^\infty  {{b_n}} } \right)\frac{{{a_{m - 1}}a_m^{ - 1}}}{{{\varpi _1}\left( {a_m^{ - 1}} \right)}}.
		\end{align*}
		
		Thirdly, note that $ {a_n}a_m^{ - 1} \in \left\{ {2k:k \in {\mathbb{N}^ + }} \right\} $. Then it follows that
		\begin{align*}
			\sin \left( {\pi {a_n}\left( {x + \Delta x} \right)} \right) &= \sin \left( {\pi {a_n}a_m^{ - 1} \cdot {a_m}\left( {x + \Delta x} \right)} \right)\\
			& = \sin \left( {\pi {a_n}a_m^{ - 1} \cdot \left( {{N_m} - {2^{ - 1}}\operatorname{sgn} \left( {{r_m}} \right)} \right)} \right)\\
			& = 0,\;\;\forall n \geqslant m + 1.
		\end{align*}
		Therefore we derive that
		\begin{align*}
			\left| {{S_3}} \right| &\leqslant \sum\limits_{n = m + 1}^\infty  {\frac{{{b_n}}}{{{\varpi _1}\left( {\left| {\Delta x} \right|} \right)}}\left| {\sin \left( {\pi {a_n}\left( {x + \Delta x} \right)} \right) - \sin \left( {\pi {a_n}x} \right)} \right|} \\
			& = \sum\limits_{n = m + 1}^\infty  {\frac{{{b_n}}}{{{\varpi _1}\left( {\left| {\Delta x} \right|} \right)}}\left| {\sin \left( {\pi {a_n}x} \right)} \right|} \\
			& \leqslant \frac{1}{{{\varpi _1}\left( {\left| {\Delta x} \right|} \right)}}\sum\limits_{n = m + 1}^\infty  {{b_n}} \\
			&\leqslant \frac{1}{{{\varpi _1}\left( {a_m^{ - 1}} \right)}}\sum\limits_{n = m + 1}^\infty  {{b_n}} .
		\end{align*}
		
		Finally by (a1), (a2) and (a3) and the estimates above, we get
		\begin{align*}
			&\frac{{\left| {f\left( {x + \Delta \left( x \right)} \right) - f\left( x \right)} \right|}}{{{\varpi _1}\left( {\left| {\Delta x} \right|} \right)}}\\
			\geqslant &\left| {{S_1}} \right| - \left| {{S_2}} \right| - \left| {{S_3}} \right|\\
			= &\frac{{{b_m}}}{{{\varpi _1}\left( {a_m^{ - 1}} \right)}} - \left( {\pi \sum\limits_{n = 1}^\infty  {{b_n}} } \right)\frac{{{a_{m - 1}}a_m^{ - 1}}}{{{\varpi _1}\left( {a_m^{ - 1}} \right)}} - \frac{1}{{{\varpi _1}\left( {a_m^{ - 1}} \right)}}\sum\limits_{n = m + 1}^\infty  {{b_n}} \\
			= &\frac{1}{{{\varpi _1}\left( {a_m^{ - 1}} \right)}}\left( {{b_m} - \left( {\pi \sum\limits_{n = 1}^\infty  {{b_n}} } \right){a_{m - 1}}a_m^{ - 1} - \sum\limits_{n = m + 1}^\infty  {{b_n}} } \right)\\
			\geqslant &\frac{1}{{{\varpi _1}\left( {a_m^{ - 1}} \right)}}\left( {{2^{ - 1}}{b_m} - \sum\limits_{n = m + 1}^\infty  {{b_n}} } \right)\\
			\geqslant &m \to  + \infty ,\;\;m \to  + \infty .
		\end{align*}
		This implies that $ f $ is $ {\varpi _2}  $ continuous but nowhere $ \varpi_1 $ continuous, i.e., $ {\varpi _1} \lesssim {\varpi _2} $, we therefore finish the proof.
	
\end{proof}
\begin{remark}
	Actually, we claim that one could construct an explicit  modulus of continuity $ \varpi^* $ of $ f(x) $ which might be weaker than the optimal one $ \varpi_2 $, i.e., $ {\varpi _1}  \lesssim {\varpi _2} \lesssim {\varpi^*} $ at this point.
	
	Namely, in view of (a1) and (a2) we obtain that
	\[{a_m}{b_m} \geqslant 2\pi {a_{m - 1}}\sum\limits_{n = 1}^\infty  {{b_n}}  \geqslant {2^m}\pi {a_0}\sum\limits_{n = 1}^\infty  {{b_n}} ,\]
	which leads to $ \sum\limits_{n = 1}^\infty  {{a_n}{b_n}}  =  + \infty  $. Define
	\[N\left( h \right): = \max \left\{ {N \in {\mathbb{N}^ + }:h\pi \sum\limits_{n = 1}^{N\left( h \right)} {{a_n}{b_n}}  \leqslant 2\sum\limits_{n = N\left( h \right) + 1}^\infty  {{b_n}} : = \frac{1}{2}{\varpi ^ * }\left( h \right)} \right\}.\]
	One can verify that  $ N\left( h \right) \to  + \infty  $ and $ \sum\limits_{n = N\left( h \right) + 1}^\infty  {{b_n}}  \to {0^ + } $ as $ h \to {0^ + } $, which implies that the modulus of continuity $ \varpi^* $ above is well defined. Therefore,
	\begin{align*}
		\left| {f\left( {x + h} \right) - f\left( x \right)} \right| &= \left| {\sum\limits_{n = 1}^\infty  {{b_n}\left( {\sin \left( {\pi {a_n}x + \pi {a_n}h} \right) - \sin \left( {\pi {a_n}x} \right)} \right)} } \right|\\
		& \leqslant \sum\limits_{n = 1}^{N\left( h \right)} { + \sum\limits_{n = N\left( h \right) + 1}^\infty  {{b_n}\left| {\sin \left( {\pi {a_n}x + \pi {a_n}h} \right) - \sin \left( {\pi {a_n}x} \right)} \right|} } \\
		& \leqslant h\pi \sum\limits_{n = 1}^{N\left( h \right)} {{a_n}{b_n}}  + 2\sum\limits_{n = N\left( h \right) + 1}^\infty  {{b_n}} \\
		&\leqslant 4\sum\limits_{n = N\left( h \right) + 1}^\infty  {{b_n}} \\
		& = {\varpi ^ * }\left( h \right).
	\end{align*}
	This gives the claim.
\end{remark}

In view of condition (\textbf{A2}), comments (4), (5) and Theorem \ref{nonowh}, we could construct an  explicit illustration with weak regularity which cannot be studied by any KAM theorem known, where the parameter is defined on a set of zero Lebesgue measure.

 \begin{example}
By Theorem \ref{nonowh} and Remark \ref{noholder}, there exist $ g_i(x) $ on $ \mathbb{R} $  with $ 1\leqslant i \leqslant4 $ such that they are nowhere H\"older continuous. Let $ \xi_0 \in [-4^{-1},4^{-1}]$ be of Diophantine's type, see \eqref{dioph}.

For $ \left( {y,x} \right) \in B\left( {0,1} \right) $ and $ \xi  \in {\left[ { - 1,1} \right]^n} \cap {\mathbb{Q}^n} $, consider the following Hamiltonian:
\[H = \sum\limits_{i = 1}^n {{\omega _i}\left( \xi  \right){y_i}}  + \varepsilon \sum\limits_{i = 1}^n {{P_i}\left( {y,\xi } \right)}, \]
where
\[{\omega _i}\left( \xi  \right) = \left\{ \begin{aligned}
	&- {2^{ - 1}} + \left( {{g_1}\left( {{\xi _i}} \right) - {g_1}\left( { - {2^{ - 1}}} \right)} \right),&{\xi _i} \in \left[ { - 1, - {2^{ - 1}}} \right] \cap \mathbb{Q} \hfill \\
	&{\xi _i},&{\xi _i} \in \left[ { - {2^{ - 1}},{2^{ - 1}}} \right] \cap \mathbb{Q} \hfill \\
	&{2^{ - 1}} + \left( {{g_2}\left( {{\xi _i}} \right) - {g_2}\left( {{2^{ - 1}}} \right)} \right),&{\xi _i} \in \left[ {{2^{ - 1}},1} \right] \cap \mathbb{Q} \hfill \\
\end{aligned}  \right.\]
and
\[{P_i}\left( {y,\xi } \right) = \left\{ \begin{aligned}
	&\left[ { - {2^{ - 1}} + \left( {{g_3}\left( {{\xi _i}} \right) - {g_3}\left( { - {2^{ - 1}}} \right)} \right)} \right]{y_i},&{\xi _i} \in \left[ { - 1, - {2^{ - 1}}} \right] \cap \mathbb{Q} \hfill \\
	&\;{\xi _i}{y_i},&{\xi _i} \in \left[ { - {2^{ - 1}},{2^{ - 1}}} \right] \cap \mathbb{Q} \hfill \\
	&\left[ {{2^{ - 1}} + \left( {{g_4}\left( {{\xi _i}} \right) - {g_4}\left( {{2^{ - 1}}} \right)} \right)} \right]{y_i},&{\xi _i} \in \left[ {{2^{ - 1}},1} \right] \cap \mathbb{Q} \hfill \\
\end{aligned}  \right.\]
for $ 1 \leqslant i \leqslant n $.

One can easily verify that $ {\varpi _1} = {\varpi _2} = x $ on $ B\left( {{\xi _0},{8^{ - 1}}} \right) $, i.e., (\textbf{A2}) holds. Neither $ \omega $ nor $ P $ is H\"older  continuous over the entire region of parameter which is a set of zero Lebesgue measure. However, the invariant torus and the prescribed frequency could be persisted through our KAM theorem, as long as $ \varepsilon>0 $ is sufficiently small.
 \end{example}

	\section*{Acknowledgement}	
The authors would like to thank the editors and referees  for their valuable suggestions and comments, which led to a significant improvement of the paper. The corresponding author (Y. Li) was supported in part by National Basic Research Program of China (Grant No. 2013CB834100), National Natural Science Foundation of China (Grant No. 12071175, Grant No. 11171132,  Grant No. 11571065), Project of Science and Technology Development of Jilin Province (Grant No. 2017C028-1, Grant No. 20190201302JC), and Natural Science Foundation of Jilin Province (Grant No. 20200201253JC).

\end{document}